\documentclass[11pt]{amsart}
\usepackage{amscd,amssymb,amsmath,enumerate}
\usepackage[all]{xy}
\usepackage{graphicx}
\usepackage{graphics}
\usepackage{latexsym}
\usepackage[all]{xy}
\usepackage{psfrag}
\xyoption{matrix} \xyoption{arrow}
\usepackage{parskip}

\newtheorem{prop}{Proposition}
\newtheorem{theorem}[prop]{Theorem}
\newtheorem{lemma}[prop]{Lemma}
\newtheorem{cor}[prop]{Corollary}

\newtheorem{definition}[prop]{{Definition}}

\newtheorem{remark}[prop]{{Remark}}

\renewcommand{\aa}

\newcommand{\End}{\operatorname{End}\nolimits}
\newcommand{\Hom}{\operatorname{Hom}\nolimits}

\newcommand{\soc}{\operatorname{soc}\nolimits}
\newcommand{\rad}{\operatorname{rad}\nolimits}

\renewcommand{\dim}{\operatorname{dim}\nolimits}
\newcommand{\gldim}{\operatorname{gldim}\nolimits}
\newcommand{\repdim}{\operatorname{repdim}\nolimits}

\usepackage{color}
\definecolor{candyapplered}{rgb}{1.0, 0.03, 0.0}
\makeatletter
\def\thm@space@setup{%
  \thm@preskip=0.7cm \thm@postskip=0.3cm
}
\makeatother

\begin{document}

\date{\today}
\title[Finitistic dimension of special multiserial algebras]{On the representation dimension and finitistic dimension of special multiserial algebras}

 \dedicatory{Dedicated to Ed Green on the occasion
     of his $70$th birthday}

\author{Sibylle Schroll}
\address{Department of Mathematics, University of Leicester, University Road, Leicester, LE1 7RH, United Kingdom.}
\email{schroll@leicester.ac.uk}

\keywords{representation dimension, radical embedding, splitting daturm, special multiserial algebra}
\thanks{Part of this  work took place during a visit of the author to the University of S\~ao Paulo: the author would like to thank Eduardo Marcos for his hospitality. This work was supported through the EPSRC fellowship grant EP/P016294/1. }

\subjclass[2010]{16G10, 05E10}
%05E10: Combinatorial aspects of representation theory
%16G10: Representations of Artinian rings
%18E30: Derived categories, triangulated categories

\begin{abstract}
For monomial special multiserial algebras, which in general are of wild representation type, we construct radical embeddings into  algebras of finite representation type. As a consequence, we show that the representation dimension of monomial and self-injective special multiserial algebras is less than or equal to three. This implies that  the finitistic dimension conjecture holds for all special multiserial algebras.
\end{abstract}

\maketitle

%{\small
%\setcounter{tocdepth}{1}
%\tableofcontents
%}

%=========================================================================================
% SECTION
\section*{Introduction}
%=========================================================================================

Many of the important open conjectures in representation theory of Artin algebras are of a homological nature, such as  the finitistic dimension conjecture, Nunke's condition and  Nakayma's conjectures. Amongst these conjectures there is a logical hierarchy, in that if  the finitistic dimension conjecture holds then Nunke's condition holds which in turn implies the  Nakayama conjectures; for an overview, see, for example \cite{Happel, Schroer, ZH}. 

The finitistic dimension conjecture states that for any Artin algebra $A$,  the supremum of the projective dimensions of the finitely generated right $A$-modules of finite projective dimension  is finite. This conjecture was originally posed as a question by Rosenberg and Zelinsky and then published by Bass in 1960 \cite{Bass}.

Although the finitistic dimension conjecture is open in general, there has been much related work in recent years reducing the problem to simpler classes of algebras \cite{  Wei, X}. There are many classes of algebras where the conjecture has been shown to hold \cite{EHIS, IT}.  For classes of algebras of mostly wild representation type, the two most prominent examples where the finitistic dimension conjecture is known to hold are the monomial algebras   \cite{GKK, IZ} and the radical cubed zero algebras \cite{GZ}. 

In this paper, we will show that the finitistic dimension conjecture holds for special multiserial algebras, a large class of mostly wild algebras, containing many other important and well-studied classes of algebras such as, for example, special biserial algebras, symmetric radical cubed zero algebras and almost gentle algebras \cite{GS1, GS2}. 

It is well known that most finite dimensional algebras are of wild representation type 
implying that their representation theory is at least as complicated as the representation
 theory of the free associative algebra in two variables. Special multiserial algebras 
 form a class of mostly wild finite dimensional algebras. It was shown in \cite{GS1} that the radical of their indecomposable modules  is a sum of uniserial modules 
 whose pairwise intersection is either a simple module or zero. This is an indication that uniserial modules play an important role in the study of their representation theory. In this paper, we show that for a monomial special multiserial algebra $A$ of infinite representation type,  the direct sum of all uniserial submodules of $A$ gives rise to an Auslander generator of $A$. 
 
  In order to show this, we construct radical embeddings from monomial special multiserial algebras to a direct product of representation finite string algebras whose quivers are  linearly oriented Dynkin diagrams of type $\mathbb{A}$ and cyclically oriented Dynkin diagrams of type $\mathbb{\tilde{A}}$. 
  Therefore by \cite{EHIS} we obtain that the representation dimension of a monomial special multiserial algebra is less or equal to three.

We further show that for any special multiserial algebra  $A$, a relation is either monomial or is a linear combination of elements in the socle of $A$.  We then apply the results in \cite{EHIS}  in combination with  our results on monomial special multiserial algebras, to show that the representation dimension of self-injective special multiserial algebras is less or equal to three.

To summarise, in this paper we show the following:
 
\begin{theorem}\label{mainresult}
Let $A$ be a monomial special multiserial algebra. Then there exists a radical embedding 
$f: A \to B$ where $B$ is an algebra of finite representation type.   
\end{theorem}

\begin{cor}
Let $A$ be a monomial special multiserial algebra. Then $repdim (A) \leq 3$.
\end{cor}

In \cite{GS3} Brauer configuration algebras are defined as generalisations of Brauer graph algebras. Brauer configuration algebras are symmetric algebras, so in particular they are self-injective and it follows from the next result that their representation dimension is less or  equal to 3. 

\begin{cor}
Let $A$ be a self-injective special multiserial algebras.  Then $repdim (A) \leq 3$. In particular, the representation dimension of a Brauer configuration algebra is  less or equal to $3$.
\end{cor}

\begin{cor}\label{FinDim} Let $A$ be a special multiserial algebra. Then the finitistic dimension of $A$ is finite.  
\end{cor}

{\bf Acknowledgements:} I would like to thank Ed Green for  the many wonderful hours we have spent talking mathematics. In particular, I would like to thank him, Karin Erdmann and Jan Geuenich for the discussions involving the results of this paper. I also would like to thank Julian K\"ulshammer for alerting me to a mistake in an earlier version of the paper.

%=========================================================================================
% SECTION
\section{Background}
%=========================================================================================

Let $K$ be an algebraically closed field. A quiver $Q = (Q_0, Q_1, s, e)$ consists of a finite set of vertices $Q_0$, a finite set of arrows $Q_1$ and maps $s, e : Q_1 \to Q_0$ where, for $a \in Q_1$, $s(a)$ denotes the vertex at which $a$ starts and $e(a)$ denotes the vertex at which $a$ ends. For $a, b \in Q_1$, such that $e(a) = s(b)$, we write $ab$ for the element in $KQ$ given by the concatenation of $a$ and $b$. For $v \in Q_0$, denote by $e_v$ the associated idempotent. We call an element $x \in KQ$ uniform if there exists $v, w \in Q_0$ such that $e_v x  e_w = x$. All modules considered are finitely generated right modules and  for $A$ a finite dimensional $K$-algebra, we denote by $\mod A$ the category of finitely generated right $A$-modules. Furthermore, set $D(A) = \Hom_K(A,K)$ and denote by $J_A$ the Jacobson radical of $A$. We call a finite dimensional $K$-algebra basic, if $A = KQ/I$ for $I$ an admissible ideal in $KQ$. 

From now on, whenever we write $A= KQ/I$, we assume that $I$ is admissible.

Recall that $\gldim (A) = {\rm sup} \{{\rm pd} (M) \mid M \in \mod A\}$  %where $pd(M)$ is the projective dimension of $M$, 
and   that $M$ is a generator-cogenerator of $A$ if $A \oplus D(A) \in {\rm add }  \; M$ where ${\rm add }  \; M$ is the subcategory of $\mod A$ generated by direct sums of direct summands of $M$. Then $$\repdim (A) = {\rm inf} \{ \gldim (\End_A(M)^{op}) \mid M \mbox{ is a generator-cogenerator of $A$ } \}.$$  Moreover, $M$ is an {\it Auslander generator} of $A$ if $\gldim (End_A(M)^{op}) = \repdim (A)$. 
The {\em finitistic dimension of $A$} is given by $${\rm findim}(A) = {\rm sup} \{ {\rm pd}(M) \mbox{ for all $M$ such that } {\rm pd}(M) < \infty\}.$$

Let $A = KQ/I$, we say that condition (S) holds for $A$ if the following holds: 

\medskip

 (S) For all  $a \in Q_1$ there exists at most one arrow $b \in Q_1$ such that $ab \notin I$ and there exists at most one arrow $c \in Q_1$ such that $ca \notin Q_1$.

\medskip

\begin{definition} {\rm A finite dimensional algebra $A$ is {\it special multiserial} if it is Morita equivalent to an algebra $KQ/I$ such that (S) holds. }
\end{definition}

We recall the following results and definitions  from \cite{EHIS}. 
For $v \in Q_0$, set $S(v)$ to be the subset of $Q_1$ consisting of arrows starting at $v$ and set $E(v)$ to be  the set of arrows of $Q_1$ ending at $v$. Note that if there is a loop $a$ at $v$ then $a \in E(v) \cap S(v) \neq \emptyset$. 

Suppose $S(v) = S_1 \sqcup S_2$ and $E(v) = E_1 \sqcup E_2$ are disjoint unions.  
The collection $Sp = (S_1, S_2, E_1, E_2)$ is a {\it splitting datum at $v$} (for $I$) if
\begin{enumerate} 
\item $ab \in I$, for all $a \in E_i$ and $b \in S_j$ with $i \neq j$, 
\item $I = \langle \rho \rangle $ where $\rho$ is a set of relations of the form $\sum \lambda apb $ such that none of the $a$ are in $E_1$ or none of the $a$ are in $E_2$ and such that  none of the $b$ are in $S_1$ or none of the $b$ are in $S_2$. 
\end{enumerate} 
 
An algebra $KQ/I$ is called {\it monomial} if $I$ is monomial, that is if $I$ is generated by paths.  Remark that condition (2) always holds if $I$ is monomial.

 Let $Sp = (S_1, S_2, E_1, E_2)$ be a splitting datum at $v$. Then we define a new quiver $$Q^{Sp} = ( Q_0^{Sp}, Q_1^{Sp}, s^{Sp}, e^{Sp} )$$ by setting $$Q_0^{Sp} = \{ v_1, v_2 \} \cup Q_0 \setminus \{ v\}$$ and $$Q_1^{Sp} = Q_1.$$ The map $s^{Sp}: Q_1^{Sp} \to Q_0^{Sp}$ is given by
 \begin{equation*}
  s^{Sp} (a) = 
    \begin{cases}
       v_i  & \mbox{ if } a \in S_i, i = 1,2, \\
       s(a) & \mbox{ otherwise. } \\
    \end{cases}       
\end{equation*}  

  The map $e^{Sp}: Q_1^{Sp} \to Q_0^{Sp}$ is given by
 \begin{equation*}
  e^{Sp} (a) = 
    \begin{cases}
       v_i  & \mbox{ if } a \in E_i, i = 1,2, \\
       e(a) & \mbox{ otherwise. } \\
    \end{cases}       
\end{equation*}

 We define $A^{Sp} = KQ^{Sp}/ I^{Sp}$ for $I^{Sp} = \langle \rho^{Sp} \rangle$ where
 
 $$\rho^{Sp} = \rho \setminus (\{ ab \vert a \in E_i \mbox{ and } b \in S_j, \mbox{ for } i \neq j   \} $$ 
  
 A {\it radical embedding} $f : A \to B$ is an algebra monomorphism such that $f (J_A) = J_B$. 
It is shown in \cite{EHIS} that  a  splitting datum gives rise to a radical embedding.  

\begin{prop}\cite{EHIS}\label{EHISsp}
Let $A = KQ/I$ with $I$ admissible. Let $Sp = (E_1, E_2, S_1, S_2)$ be a  splitting datum at some vertex $v$ of $Q$. Then there exists a radical embedding $f: A \to A^{Sp}$. 
\end{prop}

Also recall the following results from \cite{EHIS}.

\begin{theorem}\cite{EHIS}\label{EHIS}  Let $A$ and $B$ be basic algebras.  \begin{enumerate} 
\item 
If $f: A \to B$ is a radical embedding with $B$ a representation finite algebra then $repdim (A) \leq 3$. 

\item Let  $P$  be an indecomposable projective-injective $A$-module and set   $ A/ soc(P)$. Then $repdim(A) \leq 3$ if $repdim (A/ soc(P)) \leq 3$. 
\end{enumerate}
\end{theorem}

\section{Some results on special multiserial algebras}

In the following proposition we show that the relations in a special multiserial algebra are of a particular form.

\begin{prop}\label{prop:selfinjective}
   Let $A = KQ/I$ be a special multiserial algebra with $I = \langle \rho \rangle$ satisfying condition (S). Let
$r = \sum \lambda_p p  \in \rho$ be   uniform with $\lambda_p \in K$ such that almost all $\lambda_p = 0 $ and where each $p$ is a path in $Q$. Then either $r$ is a path or for all $\lambda_p \neq 0$, $p$ is in the socle of $A$ as a right and left $A$-module.   
\end{prop}

\begin{proof}
We will start by showing that the result holds when considering the socle of $A$ as a right $A$-module. 
Suppose there exists a unique $\alpha_p \neq 0$, then $r = p$. 

Suppose that $r = \lambda_p p - \lambda_q q$ with $p, q \notin I$. Then without loss of generality we can assume that $\lambda_p =1$. 
\newline
Now suppose that $p \notin \soc A$ and that $q \in \soc A$. Then there exists $a \in Q_1$ such that $pa \notin I$ but since $q \in \soc A$, $qa \in I$ and this a contradiction. Suppose now that
$p, q \notin \soc A$. Then there exist $a, b \in Q_1$ such that $pa, qb \notin I$. Since $p-\lambda_q q \in I$, by condition (S) we have $a=b$.  Therefore if 
$p = p'c$ and $q = q'd$ for $c, d \in Q_1$ then $ca, da \notin I$. This implies  by condition (S) that $c=d$ and hence $(p' - \lambda_q q')c \in I$. Moreover, $p'c, q'c \notin I$.
Now let $p' = p'' c'$ and $q' = q'' d'$ which implies that $c'c, d'c \notin I$ and $c' = d'$. Continuing in this way, we see that $p = q$. 

Suppose now that $r = \sum \lambda_p p$ and suppose that $\lambda_q \neq 0$ and $q \notin \soc A$. Then there exists $a \in Q_1$ such that $qa \notin I$ and therefore there exists $q'$ with $\lambda_{q'} \neq 0$ and $q' a \notin I$. Since $A$ is special, this implies $q = q'$. Inductively it then follows that $pa \notin I$ for any $p$ such that $\lambda_p \neq 0$ and using that $A$ is special it follows that  $r =  q$. 

Note that we have only used specialness on the right  side. Using specialness on the left side, we obtain the result for the socle of $A$ as a left $A$-module. 
\end{proof}

The following follows directly from condition (S).

\begin{lemma}\label{AisSpecial}
Let $A = KQ/I$ be monomial special multiserial and let $Sp = (S_1, S_2, E_1, E_2 )$ be a splitting datum at some vertex $v$ in $Q$. 
\begin{enumerate} \item Suppose that  $S_1 = \{b\}$, for  $ b \in Q_1$. Then $E_1$ consists of the unique arrow $a$ such that $ab \notin I$ if such an arrow $a$ exists, otherwise $E_1$ is empty. 
\item Suppose that $E_1 = \{c\}$, for $c \in Q_1$  then $S_1 = \{d\}$ where $d$ is the unique arrow such that $cd \notin I$ if such an arrow $d$ exits and $S_1$ is empty otherwise. 
\end{enumerate}
Moreover, for $Sp$ as in (1) or (2) above,  $A^{Sp}$ is monomial special multiserial.  
\end{lemma}

%=========================================================================================
% SECTION
\section{Proof of Theorem 1}
%=========================================================================================

We show that for any monomial special multiserial algebra $A = KQ/I$ there is a radical embedding of $A$ into a disjoint union of representation finite string algebras whose underlying quiver is either a linearly oriented quiver of type ${\mathbb A}$ and  or a cyclically oriented quiver of type $\widetilde {\mathbb A}$.

{\it Proof of Theorem 1 and Corollary 2:} 
Let $A = KQ/I$ be a monomial special multiserial algebra such that $I$ is generated by  paths. 
Define \sloppy $c(A) = \vert \{ v \in Q_0 \vert S(v) >1\} \vert + \vert \{  v \in Q_0 \vert E(v) >1 \} \vert$.

If $c(A) = 0$ then $Q$ is a disjoint union of quivers where each quiver  is either a linearly oriented quiver of type ${\mathbb A}$ or a cyclically oriented quiver of type $\widetilde {\mathbb A}$. So $A$ is a product  of representation finite string algebras,  and it therefore is  of finite representation type. 

Suppose that $c(A) \geq 1$. Let $v \in Q_0$ such that $\vert S(v) \vert > 1$ or $\vert E(v) \vert > 1$. Suppose that $ S(v) = \{b_1, \ldots, b_n \}$ with $n \in \mathbb N, n > 1$. Set
$$ \begin{array}{lll}  
  	S_1 & = & \{ b_1\}, \\ 
  	S_2 &=& \{ b_2, \ldots, b_n\}, \\
	E_2 &=& \{ a \in E(v) \vert ab_1 \in I \}, \\
	E_1 &=& E(v) \setminus E_2. 
\end{array}$$
Note that $E_1$ consists of the unique arrow $a \in Q_1$ such that $ab_1 \notin I$ if such an arrow exists. 
That   $Sp = (S_1, S_2, E_1, E_2) $ is a splitting datum at $v$  follows directly (S) and from the fact that $I$ is monomial. 

By Lemma~\ref{AisSpecial}, $A^{Sp}$ is again a monomial special multiserial algebra and $c(A^{Sp}) \leq c(A) -1$. 

We treat the case $\vert E(v) \vert > 1$ in a similar way.

Repeating this a finite number of times and setting  $A = A_1$ and $A_2 = A^{Sp}$, we obtain by 
 Proposition~\ref{EHISsp}  a sequence of radical embeddings
  $A_1 \to A_2  \to \cdots A_k = B$ such that $B$ is a string algebra with $c(B)=0$ and $B$ is therefore representation finite. 
Then it follows from Theorem~\ref{EHIS} (1) that $repdim (A) \leq 3$. 
 $\Box$

Let $f: A \to B$ be the radical embedding constructed in the proof of Theorem 1 above. By the proof of Theorem 1.1 in \cite{EHIS} an Auslander generator of $A$ is given by $A \oplus D(A) \oplus N$ where $N$ is the direct sum of  isomorphism class representatives of the indecomposable $B$-modules considered as $A$-modules. Therefore in the case of a monomial special multiserial algebra $N$ is given by the direct sum of all uniserial submodules of $A$. $\Box$

{\it Proof of Corollary 3:}
Since $A$ is self-injective, every projective is injective. Applying Proposition~\ref{prop:selfinjective} we obtain that iteratively factoring out the socles of the projective injective indecomposable modules gives rise to a monomial special multiserial algebra. Thus the result follows from Theorem 1 and by  the successive application of Theorem~\ref{EHIS} (2). 
$\Box$

{\it Proof of Corollary~\ref{FinDim}:} It follows from Proposition~\ref{prop:selfinjective} that $A/\soc(A)$ is a monomial special multiserial algebra and the result follows from  \cite[4.3]{X} and Theorem~\ref{mainresult}.  $\Box$

%{\tt   Proposition 4.3. Let   $f: B \to A$ be a surjective homomorphism between two algebras $B$ and $A$. Suppose that the kernel of $f$ is contained in $soc(B_B)$. If repdim $(A) \leq 3$ then the finitistic dimension of $B$ is finite. }

%===============================================================
% Bibliography
%===============================================================

\end{document}